\documentclass{amsart}
\usepackage{amssymb,enumerate,stmaryrd}

\usepackage[scr=boondox]  
           {mathalpha}
           
\usepackage{tikz-cd}
\usepackage{hyperref}
\hypersetup{%
  bookmarksnumbered=true,%
  colorlinks=true,%
  linkcolor=blue,%
  citecolor=blue,%
  filecolor=blue,%
  menucolor=blue,%
  urlcolor=blue,%
  bookmarksopen=true,%
  bookmarksdepth=2,%
  pageanchor=true}

\usepackage{mathtools}
\usepackage{todonotes}

\numberwithin{equation}{section}

\usepackage{stmaryrd} 
\usepackage{extarrows} 
\usepackage{enumitem} 

\theoremstyle{plain}
\newtheorem{theorem}{Theorem}[section]
\newtheorem{prop}[theorem]{Proposition}
\newtheorem{lemma}[theorem]{Lemma}
\newtheorem{cor}[theorem]{Corollary}

\newtheorem*{prop*}{Proposition}
\newtheorem*{theorem*}{Theorem}

\theoremstyle{definition}
\newtheorem{defn}[theorem]{Definition}

\newtheorem{example}[theorem]{Example}
\newtheorem{remark}[theorem]{Remark}

\newtheorem{notation}[theorem]{Notation}

\newcommand{\m}{\mathfrak{m}}

\newcommand{\res}{\xlongrightarrow{\simeq}}

\newcommand{\flatdim}{\mathrm{flatdim}}
\newcommand{\projdim}{\mathrm{projdim}}
\newcommand{\curv}{\mathrm{curv }}
\newcommand{\injcurv}{\mathrm{inj curv }}
\newcommand{\cx}{\mathrm{cx }}
\newcommand{\injcx}{\mathrm{inj cx }}
\newcommand{\Ext}{\mathrm{Ext }}
\newcommand{\Tor}{\mathrm{Tor }}
\newcommand{\rank}{\mathrm{rank }}
\newcommand{\Hom}{\mathrm{Hom}}

\newcommand{\RHom}{\mathsf{R}\mathrm{Hom}}
\renewcommand{\L}{\mathsf{L}}

\newcommand{\Mod}{\mathsf{Mod \,R}}
\newcommand{\ModA}{\mathsf{Mod \,A}}
\newcommand{\ModB}{\mathsf{Mod \,B}}

\renewcommand{\mod}{\mathsf{mod \, R}}
\newcommand{\modS}{\mathsf{mod \, S}}

\newcommand{\D}{\mathsf{D(R)}}
\newcommand{\DB}{\mathsf{D(B)}}
\newcommand{\DA}{\mathsf{D(A)}}

\newcommand{\Dpos}{\mathsf{D_\sqsupset (R)}}
\newcommand{\DposA}{\mathsf{D_\sqsupset (A)}}
\newcommand{\DposB}{\mathsf{D_\sqsupset (B)}}

\newcommand{\Dneg}{\mathsf{D_\sqsubset (R)}}
\newcommand{\Db}{\mathsf{D_\square (R)}}

\newcommand{\Dposf}{\mathsf{D_\sqsupset^f (R)}}
\newcommand{\DposfA}{\mathsf{D_\sqsupset^f (A)}}
\newcommand{\DposfS}{\mathsf{D_\sqsupset^f (S)}}
\newcommand{\DposfB}{\mathsf{D_\sqsupset^f (B)}}
\newcommand{\DposfSm}{\mathsf{D_\sqsupset^f (S/\m S)}}

\newcommand{\Dnegf}{\mathsf{D_\sqsubset^f(R)}}
\newcommand{\DnegfA}{\mathsf{D_\sqsubset^f(A)}}
\newcommand{\DnegfS}{\mathsf{D_\sqsubset^f(S)}}
\newcommand{\DnegSm}{\mathsf{D_\sqsubset^f (S/\m S)}}

\newcommand{\Dbf}{\mathsf{D_\square^f(R)}}
\newcommand{\DbfA}{\mathsf{D_\square^f(A)}}
\newcommand{\DbfS}{\mathsf{D_\square^f(S)}}
\newcommand{\DbfSm}{\mathsf{D_\square^f (S/\m S)}}

\newcommand{\jmap}{\mathrm{j}}

\begin{document}

\title{Betti and Bass numbers of \linebreak relative Frobenius maps}

\author[Judd]{Jenna L. Judd}
\address{
Department of Mathematics\\
University of Utah\\ 
Salt Lake City, UT 84112\\ 
U.S.A.}
\curraddr{
Department of Mathematics and Statistics\\
University of Nevada, Reno\\ 
Reno, NV 89557\\ 
U.S.A.}

\begin{abstract}
Motivated by refinements of Kunz's theorem involving the growth of Betti numbers, we study Betti and Bass numbers associated to the relative Frobenius. Under appropriate hypotheses on a local ring homomorphism $\varphi \colon R \to S$, we obtain bounds on the growth of the Betti and Bass numbers of relative Frobenius pushforwards of homologically finite complexes in terms of the corresponding invariants of the residue field of the closed fiber $S/\mathfrak m S$. These results extend previous work on Betti numbers to coefficients and provide a dual perspective via Bass numbers, contributing to the study of how Frobenius actions encode information about ring homomorphisms.   
\end{abstract}

\date{\today}



\maketitle

\section{Introduction}
In commutative algebra, there is a long tradition of studying algebraic objects such as rings and modules through numerical invariants that encode structural information. Many such invariants arise from homological constructions and provide effective tools for detecting algebraic properties. Classical examples include dimension, depth, projective dimension, and injective dimension. These invariants play a central role in identifying important classes of rings, such as Cohen-Macaulay, Gorenstein, and regular rings.

This paper focuses on two homological invariants: Betti numbers and Bass numbers.  For a module over a local ring, Betti numbers measure the size of a minimal free resolution, while Bass numbers record the size of a minimal injective resolution. These invariants extend naturally to complexes in the derived category and retain their structural significance in that setting. Section 2 recalls basic properties of Betti and Bass numbers of homologically bounded complexes over Noetherian local rings, together with results on their asymptotic behavior. In particular, we review work of Avramov from the 1990s and of Avramov, Hochster, Iyengar, Miller and Yao from the early 2000s, including characterizations of regularity in terms of the growth of Betti numbers.

A central theme of this paper is the interaction between homological invariants and ring homomorphisms. Our primary focus is the Frobenius endomorphism, which plays a fundamental role in commutative algebra and algebraic geometry. The importance of the Frobenius map is already evident in the study of singularities, where many techniques depend heavily on the characteristic of the base field. In 1964, Hironaka \cite{Hir64} proved that algebraic varieties in characteristic zero admit resolutions of singularities. In contrast, resolutions of singularities in positive characteristic are known not to exist in general, and the Frobenius endomorphism serves as an important tool in studying singular behavior. A classical result of Kunz \cite{Kun69} from 1969 shows that a Noetherian ring of prime characteristic is regular if and only if its Frobenius endomorphism is flat.

Subsequent work has produced refinements of Kunz's theorem that involve modules or complexes with a Frobenius action. A notable result, proved in 2012 by Avramov,  Hochster, Iyengar, and Yao \cite{AHIY12}, states that a Noetherian local ring is regular if the Frobenius pushforward of any bounded complex of modules with finitely generated homology is flat. Their argument relies on an analysis of the growth of Betti numbers associated to such complexes. An exposition of this result and some of its consequences appear in Section 3.

The main goal of this paper is to study Betti and Bass numbers associated to the relative Frobenius. Given a ring homomorphism $\varphi\colon R \to S$, the relative Frobenius homomorphism is obtained from the pushout of $\varphi$ along the Frobenius endomorphism on $R$. Results of Radu \cite{Rad92} and Andr{\'e} \cite{And93} from the early 1990s show that flatness of the relative Frobenius characterizes regularity of the map $\varphi$, providing a relative analogue of Kunz's theorem. In light of this result, it is natural to ask whether a relative version of the Avramov-Hochster-Iyengar-Yao theorem also holds.

Recent work of McDonald \cite{McD25} provides partial progress in this direction. Using techniques based on the growth of Betti numbers, McDonald proved that if a local ring homomorphism $\varphi \colon R \to S$ has finite flat dimension, then $\varphi$ is regular if and only if the relative Frobenius pushforward of any derived fiber of $\varphi$ has finite flat dimension. This generalizes a result of Dumitrescu \cite{Dum96} from 1996, which assumes that $\varphi$ is flat and shows that $\varphi$ is regular precisely when the relative Frobenius has finite flat dimension. Section 3 reviews the construction of the relative Frobenius, McDonald's theorem, and related consequences.

Section 4 contains the main results of this paper. We extend McDonald's result to coefficients and establish a dual statement involving Bass numbers. Specifically, suppose $\varphi \colon (R,\m,k) \to (S,\mathfrak{n},\ell)$ is a finite, flat, local ring homomorphism. For a homologically finite complex of $S$-modules, we obtain bounds on the growth of the Betti and Bass numbers of its relative Frobenius pushforward in terms of the corresponding invariants of the residue field $\ell$ of $S/\m S$, the closed fiber of $\varphi$. The explicit statement of the result can be found in \hyperref[jenna]{Theorem \ref*{jenna}}. Although not treated in detail here, one can relax the hypothesis to finite flat dimension and work with the derived fiber $k \otimes_R^\L S$ in place of $S/\m S$, as discussed in \hyperref[conditions]{Remark \ref*{conditions}}.

These results contribute to the broader program of understanding how homological invariants reflect the structure of ring homomorphisms and, in particular, how the Frobenius map encodes information about regularity and singularities in prime characteristic.

\subsection*{Acknowledgments}
My advisor, Srikanth Iyengar, has been an invaluable source of guidance and support in the various stages of this endeavor. I am deeply grateful for his patience, warmth, and encouragement. Many of the ideas in this paper were shaped through conversations with him and Peter McDonald, and I thank them both for their time and helpful insights. \\~\\
This work is partly supported by National Science Foundation grant DMS-200985.

\vfill

\section{Homological Invariants over Local Rings}

In this paper we explore the growth of homological invariants along Frobenius maps. To provide context, we first recall basic definitions and properties of such invariants over general local rings, as well as their asymptotic behavior. Our primary references for this material are \emph{Cohen-Macaulay rings} \cite{BH98} by Brunz and Herzog, \emph{Derived Category Methods in Commutative Algebra} \cite{CFH24} by Christensen, Foxby, and Holm, and \emph{Infinite Free Resolutions} \cite{Avr98} by Avramov.

\subsection*{Derived Category}

We begin by establishing the homological framework of our study, recalling in particular the notion of the derived category.

Let $(R, \m, k)$ be a Noetherian, local, commutative ring and let $\Mod$ denote the category of $R$-modules. Its full subcategory consisting of finitely generated $R$-modules is denoted $\mod$. 

The \emph{derived category} of $R$, denoted $\D$, is the category obtained by inverting quasi-isomorphisms in the category of chain complexes over $\Mod$. See \cite[\S 6.4]{CFH24} for details of its construction. Any $R$-module can be viewed as a complex in $\D$ via the full embedding $\Mod \hookrightarrow \D$. 

We write $\Dpos$ for the subcategory of $\D$ whose complexes $A$ have the property that $H_n(A)=0$ for all $n \ll 0$. Similarly, $\Dneg$ denotes the subcategory of $\D$ whose complexes have the property that $H_n(A) =0$ for all $n \gg 0$. The corresponding category for complexes with bounded homology on the right and the left is denoted $\Db$.  

In what follows, we focus on the full subcategories $\Dposf, \Dnegf$ and $\Dbf$ of $\Dpos$, $\Dneg$ and $\Db$, respectively, for which the complexes $A$ have $H_n(A)$ finitely generated for all $n$. 

\subsection*{Betti Numbers}

One of the primary invariants studied in this paper are Betti numbers, which, for a module over a local ring, measure the size of a minimal free resolution. These invariants extend naturally to the derived setting, where they exhibit analogous behavior. We recall their definition in this broader context.

\begin{defn}\cite[Definition 16.4.14]{CFH24}
Let $M \in \Dposf$. The $n^{th}$ \emph{Betti number of $M$ over $R$} is
\[
\beta_n^R(M) := \rank_{k}\Tor_n^R(k, M).
\] 
Each $\beta_n^R(M)$ is finite since $M \in \Dposf$. The \emph{Poincar{\'e} Series of $M$ over $R$} is the formal power series
\[
P_M^R(t) := \sum_{n=0}^{\infty} \beta_n^R(M)t^n \in \mathbb{Z} \llbracket t \rrbracket[t^{-1}].
\]
\end{defn}

\begin{remark} Let $F \res M$ be a minimal free resolution of $M$ over $R$. A proof of the existence of such a resolution can be found in \cite[B.63]{CFH24}. The $n^\text{th}$ Betti number of $M$ can also be computed as $\rank_R F_n$. Indeed, applying the functor $k \otimes_R -$ to $F$ yields a complex with zero differential, and so
\[
\Tor_n^R(k, M) \cong H_n(k \otimes_R F) \cong k \otimes_R F_n.
\]
Hence,
\[
\rank_{k} \Tor_n^R(k, M) =\rank_k (k \otimes_R F_n).
\]
\end{remark}

\begin{remark} Fix $M, N \in \Dposf$. If $F_M \res M$ and $F_N \res N$ are minimal free resolutions of $M$ and $N$ over $R$, respectively, then $F_M \otimes_R F_N$ is a minimal free resolution of $M \otimes_R^\L N$, where $\otimes_R^\L$ is the derived tensor product (see \cite[\S 12.2]{CFH24}). We have $M \otimes_R^\L N \in \Dposf$ since $M, N \in \Dposf$ (see \cite[Proposition 12.2.10]{CFH24}). The Betti numbers of $M \otimes_R ^\L N$ are thus given by 
\[
\beta_n^R(M \otimes_R^L N) = \sum_{i+j = n} \beta_i^R(M) \cdot \beta_j^R(N).
\]
Hence, the Poincar{\'e} Series of $M \otimes_R^\L N$ over $R$ can be computed as the product of the Poincar{\'e} Series of $M$ and $N$ over $R$
\[
P_{M \otimes_R^\L N}^R(t) = P_M^R(t) \cdot P_N^R(t).
\]
See \cite[Proposition 16.4.17]{CFH24} for a proof of the above equality.
\end{remark}

\subsection*{Bass Numbers}

Dual to Betti numbers are the Bass numbers, the second invariant of interest in this paper. They measure the size of minimal injective resolutions, and we recall their definition in context of the derived category.

\begin{defn}\cite[Definition 16.4.29]{CFH24}
Let $M \in \Dnegf$. The $n^\text{th}$ \emph{Bass number of $M$ over $R$} is defined as 
\[
\mu_R^n(M) := \rank_k \Ext_R^n (k, M).
\]
\end{defn}

\begin{remark}
The $n^\text{th}$ Bass number $\mu_R^n(M)$ of $M$ counts the number of copies of the injective hull $E_R(k)$ of $k$ in $I^n$ where $M \res I$ is a minimal injective resolution $M$ (see \cite[Theorem 16.4.38]{CFH24}). 
\end{remark}

\subsection*{Curvature and Complexity} 

In general, it is difficult to obtain information about the sequences of Betti and Bass numbers. To address this, we study their asymptotic growth, as captured by the notions of \emph{curvature} and \emph{complexity}, introduced by Avramov. See \emph{Infinite Free Resolutions} \cite[\S 4.2]{Avr98} for an exposition on these concepts.

\begin{defn}
Let $M \in \Dposf$. The \emph{curvature of $M$ over $R$} is
\[
\curv_R \; M = \limsup_n \sqrt[n]{\beta_n^R(M)}.
\]
It is the reciprocal of the radius of convergence of the Poincar{\'e} Series, $P_M^R(t)$.
The \emph{complexity of $M$ over $R$} is 
\[
\cx_R \; M = \inf \left\{ d \in \mathbb{N} \; \middle| \; \begin{aligned} & \; \text{there exists} \; c \in \mathbb{R} \; \text{such that}  \\  &\beta_n^R (M) \leq cn^{d-1} \; \text{for all} \; n \gg 0  \end{aligned} \right\}.
\]
Replacing the Betti numbers by Bass numbers in the above formulas defines the \emph{injective curvature} $\injcurv_R \; M$ \emph{of $M$ over $R$} and the \emph{injective complexity} $\injcx_R \; M$ \emph{of $M$ over $R$}. 
\end{defn}

\begin{remark} \label{curvcxgrowth}
For any $M \in \mod$, the following hold (see \cite[Remark 4.2.3]{Avr98}):
\begin{enumerate}[label=(\roman*)]
\item $\projdim_R M < \infty \iff  \cx_R \; M =0 \iff \curv_R \; M < 1$
\item $\cx_R \; M < \infty \implies \curv_R \;M \leq 1$
\item $\cx_R \; k < \infty \iff \curv_R \; k \leq 1$
\end{enumerate}
\end{remark}

Curvature detects important structural properties of the local ring $R$, such as whether it is regular or a complete intersection. 

\begin{defn}
A local ring $R$ is \emph{regular} if its maximal ideal $\m$ is generated by a system of parameters. It is a \emph{complete intersection} if its $\m$-adic completion $\widehat{R}$ is isomorphic to a quotient of a regular local ring by an ideal generated by a regular sequence.
These concepts are discussed in detail in \cite[\S 2.2 and \S 2.3]{BH98}.
\end{defn}

\begin{remark} \label{regci}
For $M \in \Dbf$, we always have  
\begin{enumerate}[label=(\roman*)]
\item $\curv_R \; M \leq \curv_R \; k < \infty \quad \text{and} \quad \cx_R \; M \leq \cx_R \; k$
\item $\curv_R \; k < 1 \iff R \text{   is regular}$
\item $\curv_R \; k \leq 1 \iff R \text{   is a complete intersection}$
\end{enumerate}
There are analogous results when replacing curvature and complexity with injective curvature and injective complexity, respectively. These claims follow from \hyperref[curvcxgrowth]{Remark \ref*{curvcxgrowth}}, \cite[Proposition 4.2.4]{Avr98} and \cite[Theorem 8.1.2]{Avr98} for $M \in \mod$. They were later generalized to $M \in \Dbf$ by Avramov, Iyengar and Miller \cite[Proposition 7.1.3 (5) and Theorem 7.1.5 (iv)]{AIM06}.
\end{remark}

\begin{remark} \label{injineq}
One always has 
\[
\curv_R \; k = \injcurv_R \; k \quad \text{and} \quad \cx_R \; k = \injcx_R \; k
\]
since 
\[
\beta_n^R (k) = \rank_k \Tor_n^R(k, k) = \rank_k \Ext_R^n(k, k) = \mu_R^n (k).
\]
In particular, \hyperref[regci]{Remark \ref*{regci}} then gives
\[
\injcurv_R \; M \leq \curv_R \; k \quad \text{and} \quad \injcx_R \; M \leq \cx_R \; k
\]
for any $M \in \Dbf$.
\end{remark}

\begin{defn} \label{extremal}
A complex $M \in \Dbf$ is said to be \emph{extremal} if 
\[
\curv _R \; M = \curv_R \; k \quad \text{and} \quad \cx_R \;M = \cx_R \; k.
\]
It is \emph{injectively extremal} if 
\[
\injcurv _R \; M = \curv_R \; k \quad \text{and} \quad \injcx_R \;M = \cx_R \; k.
\]
Extremality was first introduced for modules in \cite[\S 2]{Avr96}. See \cite[\S 11]{AIM06} for extremality of a complex.
\end{defn}

\section{The Frobenius and Relative Frobenius}

This section provides an overview of the Frobenius endomorphism and the associated relative Frobenius maps. We recall their definitions, introduce notation, and review known results describing how these maps relate to the curvature and complexity of modules and complexes. These results allow us to recover important structural properties of the rings and ring maps in consideration. In particular, we recall Kunz's classical theorem, and outline refinements of the theorem, including a module version by Avramov, Hochster, Iyengar, and Yao \cite{AHIY12}, a relative version due to Radu \cite{Rad92} and Andr{\'e} \cite{And93}, and a partial generalization of the relative version established by McDonald \cite{McD25}.

\subsection*{The Frobenius Endomorphism}

We begin by recalling the basics of the Frobenius endomorphism and the notation we will use throughout.

Let $(R, \m, k)$ be a Noetherian, local, commutative ring of prime characteristic $p > 0$. The \emph{Frobenius endomorphism} $F\colon R \to R$ is defined by $F(r) = r^p$, which is a ring homomorphism since $\mathrm{char}(R) = p$. For any integer $e \geq 1$, the $e$-fold composition $F^e\colon R \to R$ is given by $F^e(r) = r^{p^e}$.

To distinguish the source and target of the Frobenius map, we often denote the target as the Frobenius pushforward $F_*^e R$ and write $F^e\colon R \to F_*^e R$. This convention allows us to view the target as an $R$-module via the Frobenius action. Similarly, for an $R$-module $M$, we write $F_*^e M$ for its pushforward along $F^e$, with $R$-module structure defined by $r \cdot m = r^{p^e} m$.

From here on, we assume $R$ is \emph{$F$-finite}, meaning that $F^e\colon R \to F_*^e R$ is a finite morphism for some (and hence all) $e \geq 1$; equivalently, $F_*^e R$ is finitely generated as an $R$-module. Under this assumption, if $M \in \mod$, then $F_*^e M \in \mod$ as well. In particular, if $M \in \Dposf$, then $F_*^e M \in \Dposf$.

Regularity of $R$ is characterized by the Frobenius map, by a result of Kunz.

\begin{theorem}\cite[Theorem 2.1]{Kun69} \label{kunz}
The Frobenius map $F^e \colon R \to F_*^eR$ is flat for some (or equivalently, all) $e \geq 1$ if and only if $R$ is regular. \qed
\end{theorem}

A generalization of Kunz's Theorem for modules was proved by Avramov, Hochster, Iyengar and Yao. We present this result with the language of extremality and injective extremality in the derived setting as introduced in \hyperref[extremal]{Definition \ref*{extremal}}, and show how it recovers Kunz's Theorem in \hyperref[AHIYimplieskunz]{Remark \ref*{AHIYimplieskunz}}.

\begin{theorem}\cite[Corollary 5.6 and Proposition 3.2]{AHIY12} \label{AHIY}
Let $M \in \Dbf$ and assume $M \neq 0.$ Then $F_*^eM$ is extremal and injectively extremal for $e \geq 1$.  \qed
\end{theorem}

\begin{remark}
The extremality of $F_*^eM$ is shown in \cite[Corollary 5.6]{AHIY12}. To see that $F_*^eM$ is injectively extremal, apply \cite[Proposition 3.2]{AHIY12} to $F_*^eM$ to get
\[
\injcurv_R \; F_*^eM = \curv_R \; F_*^e N \quad \text{and} \quad \injcx_R \; F_*^eM = \cx_R \; F_*^eN.
\]
where $N = \Hom_R (K^M, E) \in \Dbf$ with $E$ the injective hull of $k$ over $R$ and $K^M$ the Koszul complex on $M$. Then by \cite[Corollary 5.6]{AHIY12}, $F_*^eN$ is extremal, so $\curv_R \; F_*^e N = \curv_R \; k$ and $\cx_R \; F_*^eN = \cx_R \; k$, showing that
\[
\injcurv_R \; F_*^eM = \curv_R \; k \quad \text{and} \quad \injcx_R \; F_*^eM =\cx_R \; k.
\]
\end{remark}

The proof that $F_*^eM$ is extremal and injectively extremal shows a stronger statement that is useful for arguments made in Section 4. We state these results in the next remark for our convenience. 

\begin{remark} \label{AHIYgeq}
Proposition 4.3 in \cite{AHIY12}, an ingredient to the proof of \hyperref[AHIY]{Theorem \ref*{AHIY}}, gives the following:
\begin{enumerate}[label=(\roman*)]
\item For $M \in \Dposf$ and $M \not \simeq 0$, there are inequalities
\[
\curv_R \; F_*^e M \geq \curv_R \; k \quad \text{and} \quad \cx_ R \; F_*^e M \geq \cx_R \; k 
\]
\item For $M \in \Dnegf$ and $M \not \simeq 0$, there are inequalities
\[
\injcurv_R \; F_*^e M \geq \curv_R \; k \quad  \text{and} \quad \injcx_ R \; F_*^e M \geq \cx_R \; k 
\]
\item For $M \in \Dbf$ and $M \not \simeq 0$, equalities hold in (i) and (ii).
\end{enumerate}
\end{remark}

The next remark isolates a key part of the argument showing that \hyperref[AHIY]{Theorem \ref*{AHIY}} recovers Kunz's Theorem.

\begin{remark} \label{AHIYappl}
An implication of \hyperref[AHIY]{Theorem \ref*{AHIY}} is that
\begin{enumerate}[label=(\roman*)]
\item $R$ is regular $\iff \curv_R \; F_*^e R < 1 \iff \projdim_R \;F_*^e R < \infty$
\item $R$ is a complete intersection $\iff \cx_R \; F_*^eR < \infty \iff  \curv_R \; F_*^eR  \leq 1$ for some $e \geq 1$
\end{enumerate}
Indeed, \ref{curvcxgrowth} and \ref{regci} gives
\[
R \; \; \text{is a complete intersection} \iff \curv_R \; k \leq 1 \iff  \cx_R \; k < \infty
\]
and since $F_*^e R$ is extremal for $e \geq 1$ (\hyperref[AHIY]{Theorem \ref*{AHIY}}), we have $\curv_R \; F_*^eR = \curv_R \; k$ and $\cx_R \; F_*^eR = \cx_R \; k$. This shows part (ii). Part (i) then follows from \ref{curvcxgrowth} and \ref{regci}:
\[
R \; \; \text{is regular} \iff \curv_R \; k <1 \iff \curv_R \; F_*^e R < 1 \iff \projdim_R \; F_*^e R < \infty.
\]
\end{remark}

\begin{remark} \label{AHIYimplieskunz}
We get Kunz's Theorem from \ref{AHIY} by setting $M = R$. Indeed, assume $F^e\colon R \to F_*^e R$ is flat for $e \geq 1.$ Then $F_*^e R$ is a flat, hence projective, $R$-module, so $R$ is regular by \hyperref[AHIYappl]{Remark \ref*{AHIYappl} (i)}.
\end{remark}

\subsection*{The Relative Frobenius}

We now turn our attention to the relative version of the ideas above. As the Frobenius is used to detect properties of rings, the relative Frobenius gives structural information about underlying ring homomorphisms. We begin with the construction of the relative Frobenius map. 

\begin{defn} \label{relfrobdef}
Let $\varphi\colon (R, \m, k) \to (S, \mathfrak{n}, \ell)$ be a finite map of local, commutative rings. By taking the pushout of $\varphi$ along $F$ we get the following commutative diagram
\[
\begin{aligned}
\begin{tikzcd}[column sep=2cm]
R\arrow[r,"\varphi"]\arrow[d,"F"]&S\arrow[d]\arrow[rd,"F"]& \\
F_*R\arrow[r]&A\arrow[r,"F^{\varphi}",swap]& F_*S
\end{tikzcd}
& \quad \text{where} \quad A := S \otimes_R F_*R.
\end{aligned}
\]
The map $F^\varphi$ is called the \emph{relative Frobenius} of $\varphi$, and is given by $F^\varphi(s \otimes r) = s^pr$. 
\end{defn}

\begin{remark}
With notation as above, $A$ is a local ring with maximal ideal 
\[
m_A = \sqrt{\m A + \mathfrak{n}A}
\]
and whose residue field $k_A$ has the property that $\ell \subseteq k_A \subseteq F_*(\ell)$.
\end{remark}

\begin{example} 
Consider the map $\varphi\colon R \to R[x]$. Since $R[x] \otimes_R F_*R \cong (F_*R)[x]$, the relative Frobenius of $\varphi$ is determined by the following commutative diagram
\[
\begin{tikzcd}
R\arrow[r,"\varphi"]\arrow[d,"F"]&R[x]\arrow[d, "\alpha"]\arrow[rd,"F"]& \\
F_*R\arrow[r]&(F_*R)[x]\arrow[r,"F^{\varphi}",swap]& (F_*R)[x]
\end{tikzcd}
\]
Any polynomial $\sum_{i=1}^n r_ix^i \in R[x]$ is mapped to $\sum_{i=1}^n r_i^px^i$ under $\alpha$, so the relative Frobenius is given by 
\[
F^\varphi \left( \sum_{i=1}^n r_ix^i \right) = \sum_{i=1}^n r_ix^{pi}.
\]
\end{example}

Radu and Andr\'e established a relative version of Kunz's Theorem, which characterizes regularity of a ring map in terms of the relative Frobenius. We recall the definition of a regular map before stating the theorem.  

\begin{defn} 
A map $\varphi\colon R \to S$ of local rings is \emph{regular} if $\varphi$ is flat and geometrically regular, i.e., for all $k'$ a finite, purely inseparable field extension of $k$, $S \otimes_R k'$ is regular. 
\end{defn}

\begin{theorem}[\cite{Rad92, And93}] \label{raduandre}
The relative Frobenius $F^\varphi$ of $\varphi$ is flat if and only if $\varphi$ is regular. \qed
\end{theorem}

\begin{remark} \label{raduandreimplieskunz}
We obtain Kunz's Theorem from \ref{raduandre} by considering the relative Frobenius of $\varphi\colon \mathbb{F}_p \to R$ along $F\colon \mathbb{F}_p  \to F_*(\mathbb{F}_p)$, the Frobenius on $\mathbb{F}_p$. The Frobenius in this case is the identity on $\mathbb{F}_p$ since $x^p=x$ for all $x \in \mathbb{F}_p$. Moreover, $R \otimes_{\mathbb{F}_p} F_*(\mathbb{F}_p) \cong R$. Hence we have the following commutative diagram
\[
\begin{tikzcd}
\mathbb{F}_p\arrow[r,"\varphi"]\arrow[d,"F=\text{id}_{\mathbb{F}_p}"]&R\arrow[d]\arrow[rd,"F"]& \\
\mathbb{F}_p\arrow[r]& R \arrow[r,"F^{\varphi}",swap]& F_*R
\end{tikzcd}
\]
The map $\varphi$ is flat since its source is a field, and since $\mathbb{F}_p$ is perfect, $R$ is geometrically regular if $R$ is regular. Hence, $R$ is regular if and only if $\varphi$ is regular. This occurs if and only if $F^\varphi$ is flat, by Radu-Andr\'e (\ref{raduandre}). Here $F^\varphi$ is the Frobenius on $R$. So $R$ is regular if and only if $F\colon R \to F_*R$ is flat, which is the statement of Kunz's Theorem (\ref{kunz}). 
\end{remark}

Next, we draw a connection between the context of \hyperref[raduandre]{Theorem \ref*{raduandre}} and curvature, motivating the following theorem.

\begin{remark}\label{curvzero}
Suppose $\varphi\colon R \to S$ is a finite map of local rings. Then 
\[
F^\varphi\colon A \to F_*S \quad \text{flat} \quad \implies \quad \curv_A (F_*^\varphi S) =0.
\]
Indeed, $F^\varphi\colon A \to F_*S$ is finite since $\varphi$ is finite. So, $F^\varphi$ is flat if and only if $F_*^\varphi S$ is a free $A$-module. Hence, $F^\varphi$ is flat if and only if $\beta_n^A (F_*^\varphi S) =0$ for all $n \geq 1$, in which case $\curv_A(F_*^\varphi S) =0$. 
\end{remark}

Under appropriate hypotheses, McDonald obtains a partial generalization of the Radu-Andr\'e Theorem for modules 

\begin{theorem}\cite[Theorem 3.1]{McD25} \label{mcd}
Let $\varphi\colon R \to S$ be a flat map of $F$-finite local rings of positive characteristic. Then
\[
\curv_{S/\m S} \left(F_*(S/\m S)\right) = \curv_A(F_*^\varphi S) 
\]
provided the residue field $k$ of $R$ is perfect. \qed
\end{theorem} 

\begin{remark}
McDonald proved a more general statement in \cite[Theorem 3.1]{McD25} with the hypothesis that $\varphi$ has finite flat dimension instead of $\varphi$ being a flat map, and without the assumption that $k$ is perfect. In this generality, he proved
\[
\curv_{\bar{S}'} (F_*\bar{S'}) \leq \curv_A(F_*^\varphi S) \leq \max\{\curv_{\bar{S'}} (F_*\bar{S'}), 1 \}
\]
where $\bar{S'} = S \otimes_R^\L k'$ for $k'$ any finite, purely inseparable extension of $k$. Since we assume $\varphi$ is flat, the derived fiber is the usual fiber: $\bar{S'} \cong S \otimes_R k'$ for all $k'$. In particular, for $k'=k$, one has
\[
\curv_{S/ \m S} (F_*(S/\m S)) \leq \curv_A(F_*^\varphi S) \leq \max\{\curv_{S/ \m S} (F_*(S/\m S)), 1 \}.
\]
The equality 
\[
\curv_{S/\m S} \left(F_*(S/\m S)\right) = \curv_A(F_*^\varphi S)
\]
holds if $\curv_{S/\m S} (F_*(S/\m S)) \geq 1$. Otherwise, $\curv_{S/\m S} (F_*(S/ \m S)) <1$, that is, $\curv_{S/\m S} (F_*(S/ \m S)) = 0$. By \hyperref[curvcxgrowth]{Remark \ref*{curvcxgrowth} (i)}, this implies 
\[
\projdim_{S/\m S}(F_*(S/\m S)) < \infty.
\]
Then $S / \m S$ is regular by \hyperref[AHIYappl]{Remark \ref{AHIYappl} (i)}, and since $k$ is perfect by assumption, $S/ \m S$ is therefore regular as a $k$-algebra. Hence, the map $\varphi\colon R \to S$ is regular, so by \ref{raduandre}, $F^\varphi$ is flat. That is, $F_*^\varphi S$ is flat over $A$, so \ref{curvzero} implies $\curv_A \; F_*^\varphi S = 0$. Again, we have $\curv_{S/\m S} \left(F_*(S/\m S)\right) = \curv_A(F_*^\varphi S)$. 
\end{remark}

\begin{remark}
Assuming $\varphi\colon R \to S$ is flat and finite, McDonald's result implies \hyperref[raduandre]{Theorem \ref*{raduandre}}. If $F^\varphi$ is flat, $\curv_A (F_*^\varphi S) =0$ by \hyperref[curvzero]{Remark \ref*{curvzero}}. Then by \hyperref[mcd]{Theorem \ref*{mcd}} we have $\curv_{S/\m S} \left(F_*(S/\m S)\right)=0$. This implies $S/\m S$ is regular, by \hyperref[AHIYappl]{Remark \ref*{AHIYappl} (i)}. Since $k$ is perfect, $S \otimes_R k =S/\m S$, so $\varphi$ is geometrically regular over $k$. 
\end{remark}

As a consequence of McDonald's theorem, we can use the relative Frobenius to detect when the underlying ring map is a complete intersection. We give a simplified definition of a complete intersection map for our specific context and then make this connection explicit.

\begin{defn} \cite[\S 1]{Avr99} A flat map $\varphi\colon R \to S$ of local rings is a \emph{complete intersection} if $S/\m S$ is a complete intersection.
\end{defn}

\begin{remark}
Another consequence of McDonald's result is that, under the hypotheses of \ref{mcd}, one has
\[
\cx_A (F_*^\varphi S) < \infty \quad \implies \varphi \quad \text{is a complete intersection}
\]
Indeed, $\cx_A (F_*^\varphi S) < \infty \implies \curv_A (F_*^\varphi S) \leq 1$ by \hyperref[curvcxgrowth]{Remark \ref*{curvcxgrowth} (ii)}. Applying \ref{mcd} then gives $\curv_{S/\m S} (F_*(S/\m S)) \leq 1$. By \ref{AHIY} this implies $\curv_{S/\m S}(\ell) < 1 $, where $\ell$ is the residue field of $S$. Finally, \hyperref[regci]{Remark \ref*{regci} (iii)} gives that $S/\m S$ is a complete intersection. Since $\varphi$ is flat by assumption, it follows that $\varphi$ is a complete intersection.
\end{remark}

In the next section, we extend McDonald's result, \ref{mcd}, to coefficients.

\section{Betti and Bass Numbers over the Relative Frobenius}

In this section we present the main results of the paper, extending McDonald's theorem to coefficients and establishing an analogue for Bass numbers. Specifically, given a map $\varphi\colon R \to S$ of local rings and a homologically finite complex of $S$-modules, we obtain bounds on the curvature, complexity, and their injective analogues of its relative Frobenius pushforward in terms of the corresponding invariants of the residue field of the closed fiber of $\varphi$. We show how this recovers \hyperref[mcd]{Theorem \ref*{mcd}} due to McDonald and, under certain conditions, \hyperref[AHIY]{Theorem \ref*{AHIY}} due to Avramov, Hochster, Iyengar and Yao. We start by establishing the notation and assumptions for our main result, state the result, and then illustrate how it can be used to derive properties of $\varphi$.

Let $(R, \m, k)$ be a Noetherian, local, commutative ring of prime characteristic $p >0$. For simplicity of exposition, we suppose that $k$ is perfect. The diagram below defines the relative Frobenius $F^{e \cdot \varphi}$ of $\varphi\colon R \to S$ along the $e$-fold Frobenius $F^e$ of $R$. This notion was introduced for $e=1$ in \hyperref[relfrobdef]{Definition \ref*{relfrobdef}}.
\[
\begin{aligned}
\begin{tikzcd}[column sep=2cm]
R\arrow[r,"\varphi"]\arrow[d,"F^e"]&S\arrow[d]\arrow[rd,"F^e"]& \\
F^e_*R\arrow[r, "\psi"]&A_e\arrow[r,"F^{e \cdot \varphi}",swap]& F^e_*S
\end{tikzcd}
& \quad \text{where} \quad A_e := S \otimes_R F^e_*R.
\end{aligned}
\]
The map $F^{e \cdot \varphi}$ is given by $F^{e \cdot \varphi}(s \otimes r) = s^{p^e}r$.

In this section, we examine the growth of Betti and Bass numbers of $N \in \DposfS$ viewed as a complex over $A_e$ via the action of the relative Frobenius $F^{e \cdot \varphi}$. We assume the map $\varphi$ is flat, which ensures that $\psi$ is also flat. The main result is stated next; parts (i) and (ii) are proved in \hyperref[relfrobcurvcx]{Theorem \ref*{relfrobcurvcx}} and parts (iii) and (iv) are proved in \hyperref[injcurvcx]{Theorem \ref*{injcurvcx}}.

\begin{theorem} \label{jenna}
Suppose $\varphi\colon R \to S$ is a finite, flat map of local rings and $N \in \DbfS$ with $N \not \simeq 0$. With $\ell$ the residue field of $S$, there are inequalities for all $e \geq 1$:
\begin{enumerate}[label=(\roman*)]
\item $\curv_A ((F^{e \cdot \varphi})_* N) \geq \curv_{S/\m S} (\ell)$
\item $\cx_A((F^{e \cdot \varphi})_* N) \geq \cx_{S/\m S}(\ell)$
\item $\injcurv_A ((F^{e \cdot \varphi})_* N) \geq \curv_{S/\m S} (\ell)$
\item $\injcx_A((F^{e \cdot \varphi})_* N) \geq \cx_{S/\m S}(\ell)$
\end{enumerate}
Equality in (i) and (ii) holds, for example, when $N$ has finite flat dimension over $R$. We have equality in (iii) and (iv) if $N$ has finite injective dimension over $R$. \qed
\end{theorem}

\begin{remark} \label{conditions}
In the statement of the preceding theorem, the hypotheses that the residue field $k$ of $R$ is perfect and the map $\varphi$ is flat (rather than of finite flat dimension) are both imposed for ease of exposition. Indeed, tackling the case of a general $k$ entails only an extension of scalars to $F_*k$ in various steps of the proof. If $\varphi$ is only assumed to be of finite flat dimension, the terms on the right in parts (i)--(iv) would involve invariants of the derived fiber $S\otimes^\L_Rk$, viewed either as a simplicial, or a differential graded, algebra. The basic ingredients needed to execute the proof in this generality are already available in \cite{McD25} .
\end{remark}

\begin{cor}
With notation as in \hyperref[jenna]{Theorem \ref*{jenna}}, if there exists some $e \geq 1$ such that either $\cx_A((F^{e \cdot \varphi})_* N) < \infty$ or $\curv_A((F^{e \cdot \varphi})_* N) \leq 1$, then $\varphi$ is a complete intersection. The same statements hold when replacing complexity and curvature with their injective analogues.
\end{cor}

\begin{proof}
Suppose $\curv_A((F^{e \cdot \varphi})_* N) \leq 1$ for some $e \geq 1$. \hyperref[jenna]{Theorem \ref*{jenna}} implies that $\curv_{S/\m S}(\ell) \leq 1$ and $S/\m S$ is then a complete intersection by \hyperref[regci]{Remark \ref*{regci} (iii)}. Since $\cx_A ((F^{e \cdot \varphi})_* N) < \infty$ implies $\curv_A ((F^{e \cdot \varphi} )_*N) \leq 1$ (see \hyperref[curvcxgrowth]{Remark \ref*{curvcxgrowth} (ii)}) the proof is complete. The injective analogue is proved similarly. 
\end{proof}

\subsection*{Growth of Betti Numbers}

This section culminates in the proofs of parts (i) and (ii) of \hyperref[jenna]{Theorem \ref*{jenna}}, establishing the bounds on curvature and complexity for complexes under the relative Frobenius action. Before presenting the proof, we first examine the result from the perspective of extremality, then use it to recover results outlined in Section 3.

\begin{remark} In view of \hyperref[AHIY]{Theorem \ref*{AHIY}}, it is natural to ask if $F_*^\varphi N$ is extremal and injectively extremal for any $N \in \DbfS$. Observe that \hyperref[jenna]{Theorem \ref*{jenna}}, together with \ref{regci} and \ref{injineq}, give the following inequalities:
\begin{enumerate}[label=(\roman*)]
\item $\curv_A(k_A) \geq \curv_A(F^\varphi_*N) \geq \curv_{S/\m S} (\ell)$ 
\item $\cx_A(k_A) \geq \cx_A(F^\varphi_*N) \geq \cx_{S/\m S} (\ell)$
\item $\curv_A(k_A) \geq \injcurv_A(F^\varphi_*N) \geq \curv_{S/\m S} (\ell)$
\item $\cx_A(k_A) \geq \injcx_A(F^\varphi_*N) \geq \cx_{S/\m S} (\ell)$
\end{enumerate}
where $k_A$ is the residue field of $A$. However, the left-hand inequalities above can be strict, and we do not expect equality since asymptotic behavior of $F_*^\varphi N$ is a relative property of the map $\varphi$, not an absolute property of the ring $A$. For instance, take $\varphi\colon R \to R$ as the identity on $R$. Then $F^\varphi$ is also the identity on $R$. In this case $\curv_R(F_*^\varphi R) =0$, but $\curv_R(k)$ can be arbitrary.
\end{remark}

\begin{remark}
We recover McDonald's result \ref{mcd} from \ref{jenna} and \ref{AHIY} by setting $N = S$. For one has
\[
\curv_A(F_*^\varphi S) = \curv_{S/\m S} (\ell) =  \curv_{S/\m S} (F_* (S/\m S)).
\]
The first equality follows from \ref{jenna} since $S$ is flat over $R$ and the second equality holds since the $S/\m S$-module $F_*(S/\m S)$ is extremal.
\end{remark}

\begin{remark}
In the case where $R$ is finite over $k$, \hyperref[jenna]{Theorem \ref*{jenna}} along with \ref{regci} and \ref{injineq} recovers \hyperref[AHIY]{Theorem \ref*{AHIY}} due to Avramov, Hochster, Iyengar and Yao. The relative Frobenius of $\varphi\colon k \to R$ along $F\colon k \to k$, the Frobenius on $k$, is given by the following commutative diagram
\[
\begin{tikzcd}
k \arrow[r,"\varphi"]\arrow[d,"F" ', "\cong"]&R\arrow[d]\arrow[rd,"F"]& \\
k \arrow[r]& R \arrow[r,"F^{\varphi}",swap]& F_*R
\end{tikzcd}
\]
Since $k$ is perfect, the Frobenius on $k$ is an isomorphism. Hence, $F^\varphi$ is the Frobenius on $R$. For $M \in \Dbf$, one has
\[
\curv_R \; k \geq \curv_R (F_*M) = \curv_R (F_*^\varphi M) \geq \curv_R \; k
\]
and 
\[
\cx_R \; k \geq \cx_R (F_*M) = \cx_R (F_*^\varphi M) \geq \cx_R\; k
\]
showing that $F_*M$ is extremal. Injective extremality of $F_*M$ is shown by the inequalities
\[
\curv_R \; k \geq \injcurv_R (F_*M) = \injcurv_R (F_*^\varphi M) \geq \curv_R \; k
\]
and 
\[
\cx_R \; k \geq \injcx_R (F_*M) = \injcx_R (F_*^\varphi M) \geq \cx_R \; k.
\]
\end{remark}

Next, we establish a general result regarding Betti numbers and extension of scalars, which we will later apply in the context of \hyperref[jenna]{Theorem \ref{jenna}}. We begin by introducing the notation for this result.

\begin{notation} \label{derivedfibers}
For any map $\alpha\colon A \to B$ of commutative rings, there is an extension of scalars functor
\[ 
\ModA \to \ModB \quad \text{given by} \quad M \mapsto B \otimes_A M. 
\]
We let $\alpha^*$ denote the associated functor in the derived setting:
\[
\alpha^* \colon \DposA \to \DposB  \quad \text{given by} \quad \alpha^*M := B \otimes_A^\L M.
\]
Thus for any $M \in \DposA$, the complex $ \alpha^*M$ represents the complex $B \otimes_A G$ where $G \res M$ is a projective resolution of $M$ over $A$. There is also the restriction of scalars functor $\alpha_*(-)$. This is an exact functor, so it extends to the derived setting:
\[
\alpha_* \colon \DposB \to \DposA.
\]
\end{notation}

Let $\psi\colon R \to A$  be a flat map of local rings and let $k$ be the residue field of $R$. Base change along $R \to k$ gives the following commutative diagram
\begin{equation} \label{basechange}
\begin{aligned}
\begin{tikzcd}
R\arrow[r, "\psi"]\arrow[d]&A\arrow[d, "\alpha"] \\
k\arrow[r, "\overline{\psi}"]&B 
\end{tikzcd}
&\quad \text{where} \quad B:= A/\m A
\end{aligned}
\end{equation}

Below is the first ingredient of the proof of \hyperref[jenna]{Theorem \ref*{jenna}}.

\begin{lemma} \label{gencurvcx}
For $N \in \DposfA$, we have $\alpha^*N \in \DposfB$ and there is an equality of Betti numbers 
\[
\beta_n^A(N) = \beta_n^B(\alpha^*N) \quad \text{for all} \quad n.
\]
\end{lemma}

\begin{proof}
Let $G \res N$ be a minimal free resolution of $N$ over $A$. Since $A$ is flat over $R$ and $G \res N$ is a flat resolution over $A$, we have 
\[
\alpha^*N = B \otimes_A^\L N \cong (k \otimes_R A) \otimes_A^\L N \cong k \otimes_R^\L N \simeq k \otimes_R G.
\]
Each component $G_n$ of $G$ is a direct sum of copies of $A$, say $G_n = A^{b_n}$ for each $n$. Then each component of $k \otimes_R G$ is of the form
\[
(k \otimes_R G)_n = k \otimes_R G_n = k \otimes_R A^{b^n} = \left(A/\m A\right)^{b^n} = B^{b_n}.
\]
That is, $\alpha^*N \simeq k \otimes_R G$ is a complex of finitely generated free $B$-modules. In particular, $\alpha^* N \in \DposfB$. 

For all $n$ we have $\beta_n^A(N) = \rank_A(G_n)$. Since $G \res N$ is a minimal resolution over $A$, it follows that
\[
\partial_n(\alpha^*N) \subseteq \m_A \alpha^*N \subseteq \m_{B} \alpha^*N
\]
where $\m_A$ and $\m_B$ are the maximal ideals of $A$ and $B$, respectively. Therefore, $\alpha^*N\simeq k \otimes_R G$ is a minimal resolution over $B$, whose $n^{\text{th}}$ component is of the form $G_n/\m G_n$. So one has
\[
\beta_n^{B}(\alpha^*N) = \rank_{B} \left( G_n/ \m G_n \right) = \rank_A(G_n) = \beta_n^A(N). \qedhere
\]
\end{proof}

\begin{remark} \label{scalars}
With the goal of understanding $k \otimes_R^\L F_*^\varphi N$ in the context of \hyperref[jenna]{Theorem \ref*{jenna}}, we analyze restriction and extension of scalars for a general map. To keep the exposition clear, we approach this in three stages: beginning with the general module case, then focusing on our specific map, and finally extending to the derived category.
\begin{enumerate}[label=(\roman*)]
\item Let $\sigma\colon A \to S$ and $\alpha\colon A \to B$ be maps of commutative rings. By taking the pushout of $\sigma$ along $\alpha$ we get the following commutative diagram
\[
\begin{tikzcd}
A\arrow[r, "\sigma"]\arrow[d, "\alpha"]&S\arrow[d, "\iota"] \\
B\arrow[r, "\jmap"]&B \otimes_A S
\end{tikzcd}
\]
The restriction of scalars along $\sigma$ followed by the extension of scalars along $\alpha$ is isomorphic to the extension of scalars along $\iota$ followed by the restriction of  scalars along $\jmap$. That is, if $N \in \modS$, we have an isomorphism 
\[
\alpha^*(\sigma_* N) \cong \jmap_*(\iota^*N)
\]
of $B$-modules. 
In particular, 
\[
B \otimes_A \sigma_*N \cong \jmap_*(B \otimes_A N).
\]

\item Set $B = A/\m A = k \otimes_R A$ in the context of (i). Then 
\[
B \otimes_A S = k \otimes_R S = S/\m S
\]
so combining the diagram \ref{basechange} and the pushout of $\sigma$ along $\alpha$ gives the following commutative diagram of ring maps.
\[
\begin{tikzcd} 
R \arrow[r, "\psi"]\arrow[d] & A \arrow[r, " \sigma"] \arrow[d, "\alpha"] & S \arrow[d, "\iota"] \\
k \arrow[r, "\overline{\psi}"] &B\arrow[r, "\jmap"] & S/\m S
\end{tikzcd}
\]
As in (i), we have $B \otimes_A \sigma_*N \cong \jmap_*(B \otimes_A N)$ for $N \in \modS$. Then since $B = k \otimes_R A$, we get the following isomorphism
\[
k \otimes_R \sigma_* N = (k \otimes_R A) \otimes_A \sigma_* N \cong \jmap_* \left((k \otimes_R A) \otimes_ A N \right) = \jmap_*(k \otimes_R N).
\]
\item We obtain an analogous statement to (ii) in the derived setting. Assume $A$ and $S$ are flat over $R$. Then 
\[
B = k \otimes_R^\L A \cong k \otimes_R A \quad \text{and} \quad B \otimes_A^\L S \cong k \otimes_R^\L S \cong k \otimes_R S = S/\m S.
\]
So we obtain the same diagram as in (ii), and for $N \in \DposfS$ we have 
\[
k \otimes_R^\L \sigma_*N \simeq \jmap_* ( k \otimes_R^\L N)  \quad \text{in} \quad \DB.
\]
Using the notation for derived fibers introduced in \ref{derivedfibers}, the above line can be written as
\[
\alpha^*(\sigma_*N) \simeq \jmap_* (i^*N) \quad \text{in} \quad \DB.
\]
\end{enumerate}
\end{remark}

We establish the second key ingredient for the proof of \hyperref[jenna]{Theorem \ref*{jenna}} in the following lemma, returning to the context of the relative Frobenius of $\varphi\colon R \to S$. We take  $\sigma = F^\varphi$ and $\jmap = \overline{F^\varphi}$ and obtain the following commutative diagram
\[\
\begin{tikzcd}
R\arrow[r, "\psi"]\arrow[d]&A\arrow[r,"F^{\varphi}"]\arrow[d, "\alpha"]& S\arrow[d, "\iota"] \\
k\arrow[r, "\overline{\psi}"]&B \arrow[r, "\overline{F^\varphi}"]&S/\m S
\end{tikzcd}
\]
where  $A = S \otimes_R F_*R$ and  $B= k \otimes_R A$. 

\begin{lemma} \label{frobonS}
Let $\varphi\colon R \to S$ be a finite, flat map of local rings, $N \in \DposfS $ and assume $N \not \simeq 0$. Then 
\[
\alpha^*(F_*^\varphi N) \simeq F_*^{S/ \m S} (\iota^*N) \quad \text{in} \quad \DB
\]
where $F^{S/\m S}$ is the Frobenius on $S/\m S$.
\end{lemma}

\begin{proof}
Applying \hyperref[scalars]{Remark \ref*{scalars} (iii)} with $\sigma = F^\varphi$ and $\jmap = \overline{F^\varphi}$ yields
\[
\alpha^*(F_*^\varphi N) \simeq \overline{F^\varphi}_* (\iota^*N) \quad \text{in} \quad \DB.
\]
We claim that the $B$-action on $\iota^*N$ is via the Frobenius on $S/\m S$. Observe that 
\[
\begin{aligned}
B = k \otimes_{R} A &= k \otimes_{F_*R} (F_*R \otimes_R S) \\
&= k \otimes_R S \\
&= S/\m S.
\end{aligned}
\]
 \
 Moreover, $\overline{F^\varphi} \colon S/\m S \to S/\m S$ is the Frobenius on $S/\m S$. We see this by analyzing the diagram
\[
\begin{tikzcd}
F_*R \otimes_R S\arrow[r, "F^\varphi"]\arrow[d]&S\arrow[d] \\
k \otimes_R S \arrow[r, "\overline{F^\varphi}"]& k \otimes_R S
\end{tikzcd}
\]
An element $1 \otimes s \in F_*R \otimes_R S$ maps to itself under base change in the left-hand vertical map. It is sent to $s^p$ under the relative Frobenius $F^\varphi$, and then to $1 \otimes s^p$ under the right-hand vertical map. So $\overline{F^\varphi}$ is given by $1 \otimes s \mapsto 1 \otimes s^p$. That is, $\overline{F^\varphi} = F^{S/\m S}$. 
\end{proof}

Having assembled all the components, we are ready to present the proof.

\begin{theorem} \label{relfrobcurvcx}
Let $\varphi\colon R \to S$ be a finite, flat map of local rings, $N \in \DposfS$ and assume $N \not \simeq 0$. Then for $\ell$ the residue field of $S$, we have
\[
\curv_A((F^{e \cdot \varphi})_* N) \geq \curv_{S/\m S} (\ell) \quad \text{and} \quad \cx_A((F^{e \cdot \varphi})_* N) \geq \cx_{S/\m S} (\ell).
\]
Equality holds, for example, when $N$ has finite flat dimension over $R$.
\end{theorem}

\begin{proof}
We present the proof for $F$ instead of $F^e$ for notational simplicity. The reader is invited to think of $F$ as a power of the Frobenius, in which case the same proof applies. We combine \hyperref[gencurvcx]{Lemma \ref*{gencurvcx}} applied to $\psi\colon F_*R \to A$ and \hyperref[frobonS]{Lemma \ref*{frobonS}} with $B = S/\m S$ as above to get
\[
\beta_n^A(F_*^\varphi N) = \beta_n^B( \alpha^*(F_*^\varphi N)) = \beta_n^{S/\m S} \left(F_*^{S/ \m S} (\iota^*N) \right). 
\]
Since $N \in \DposfS$, we know $\iota^* N \in \DposfSm$. Therefore, by \hyperref[AHIYgeq]{Remark \ref*{AHIYgeq} (i)}
\[
\curv_A(F_*^\varphi N) = \curv_{S/\m S} \left(F_*^{S/ \m S} (\iota^* N)\right) \geq \curv_{S/\m S} (\ell)
\]
and 
\[
\cx_A(F_*^\varphi N) = \cx_{S/\m S} \left(F_*^{S/ \m S} (\iota^* N)\right) \geq \cx_{S/\m S} (\ell).
\]
If $\iota^* N \in \DbfSm$, then equalities hold above by \hyperref[AHIYgeq]{Remark \ref*{AHIYgeq} (iii)}. We have 
\[
\begin{aligned}
\iota^* N = k \otimes_R^\L N \in \DnegSm &\iff H_n(k \otimes_R^\L N) = 0 \quad \text{for all } \quad n \gg 0 \\
&\iff \Tor_n^R (k, N) = 0 \quad \text{for all} \quad n \gg 0 \\
&\iff \flatdim_R N < \infty.
\end{aligned} 
\] 
showing the desired equality holds when $N$ has finite flat dimension over $R$. 
\end{proof}

\subsection*{Bass Numbers}

In this section we prove an analogue of \hyperref[relfrobcurvcx]{Theorem \ref*{relfrobcurvcx}} for injective curvature and injective complexity, which covers the remaining parts of the main \hyperref[jenna]{Theorem \ref*{jenna}}. First we introduce the theory of dualizing complexes and Grothendieck Duality, using \cite[\S 18.2]{CFH24} as the main reference. In what follows, we assume $(A, \m_A, k_A)$ is a commutative, local ring.

\begin{defn}\cite[Proposition 18.2.1, Definition 18.2.21, Theorem 18.2.24]{CFH24}
 A \emph{dualizing complex} $\omega_A$ for $A$ is a complex in $\DbfA$ such that $\omega_A$ has finite injective dimension over $A$ and there is an isomorphism $\RHom_A(\omega_A, \omega_A) \simeq A$ in $\DA$. The complex $\omega_A$ is \emph{normalized} if $\RHom_A(k_A, \omega_A) \simeq k_A$.
\end{defn}

See \cite[\S 12.2]{CFH24} for the construction of the derived hom functor $\RHom$ and \cite[\S 17.3]{CFH24} for an exposition on injective dimension of complexes over local rings.

\begin{remark}
Not every ring has a dualizing complex. One can read about the conditions required for the existence of dualizing complexes in \cite[\S 18.2]{CFH24}. For our purposes, we assume that $A$ has a normalized dualizing complex $\omega_A$. 
\end{remark}

The heart of our argument involves dualizing the complex $\omega_A$, which we accomplish using the functor introduced in the following proposition.

\begin{prop}\cite[Theorem 18.2.3]{CFH24}
The Grothendieck Duality functor
\[
\RHom_A(-, \omega_A)\colon \DposfA \to \DnegfA
\]
maps $M \in \DposfA$ to $\RHom_A(M, \omega_A) \in \DnegfA$, which we abbreviate as $(-)^\dagger$. That is, for any $M \in \DposfA$, we set 
\[
M^\dagger := \RHom_A(M, \omega_A) .
\]
The map $M \to M^{\dagger\dagger}$ is an isomorphism in $\DA$. \qed
\end{prop}

The duality functor provides a natural relationship between Betti and Bass numbers, allowing us to translate the established bounds on curvature and complexity directly into bounds on injective curvature and injective complexity.

\begin{prop} \cite[Theorem 18.2.32]{CFH24}  \label{bettibassdual}
Let $M \in \DbfA$. One has
\[
\mu_A^n(M) = \beta_n^A(M^\dagger) 
\]
for all $n$. \qed
\end{prop}

Finally, we show that finite maps are compatible with the duality functor: the dual of a pushforward is the pushforward of the dual.

\begin{lemma} \label{dualizinghom}
Let  $\sigma\colon A \to S$ be a finite map. Then $\omega_S \cong \RHom_A(S, \omega_A)$ is a normalized dualizing complex for $S$ and for $N \in \DbfS$ we have an isomorphism
\[
\RHom_A(\sigma_*N, \omega_A) \cong \sigma_* \RHom_S(N, \omega_S).
\]
That is, 
\[
(\sigma_*N)^\dagger \cong \sigma_*(N^\dagger).
\]
\end{lemma}

\begin{proof}
Since $S$ is finite over $A$,  $\omega_S \cong \RHom_A(S, \omega_A)$ is a normalized dualizing complex for $S$ by \cite[Theorem 18.2.6]{CFH24}. We have
\[
\begin{aligned}
\RHom_A(\sigma_*N, \omega_A) &\simeq \RHom_A(S \otimes_S^\L N, \omega_A) \\
&\simeq  \sigma_* \RHom_S(N, \RHom_A(S, \omega_A)) \\
&\simeq \sigma_* \RHom_S(N, \omega_S)
\end{aligned}
\]
where the first quasi-isomorphism follows from $N \otimes_S^\L S \simeq N$ and the second is by derived hom-tensor adjunction \cite[Proposition 12.3.17]{CFH24}.
\end{proof}

At last, we prove the desired bounds on injective curvature and complexity. 

\begin{theorem} \label{injcurvcx}
Let $\varphi\colon R \to S$ be a finite, flat map of local rings, $N \in \DnegfS$ and assume $N \not \simeq 0$. Then for $\ell$ the residue field of $S$, we have
\[
\injcurv_A((F^{e \cdot \varphi})_* N) \geq \curv_{S/\m S}(\ell) \quad \text{and} \quad \injcx_A((F^{e \cdot \varphi})_* N) \geq \cx_{S/\m S}(\ell).
\]
Equality holds, for example, if $N$ has finite injective dimension over $R$. 
\end{theorem}

\begin{proof}
Apply \hyperref[dualizinghom]{Lemma \ref*{dualizinghom}} to $\sigma= F^{e \cdot \varphi}$. Together with \hyperref[bettibassdual]{Proposition \ref*{bettibassdual}} we have
\[
\mu_A^n\left((F^{e \cdot \varphi})_* N\right) = \beta_n^A \left(((F^{e \cdot \varphi})_* N)^\dagger \right) = \beta_n^A \left((F^{e \cdot \varphi})_* (N^\dagger)\right).
\]
Applying \hyperref[relfrobcurvcx]{Theorem \ref*{relfrobcurvcx}} to $N = N^\dagger$ gives
\[
\injcurv_A((F^{e \cdot \varphi})_* N) = \curv_A\left((F^{e \cdot \varphi})_* (N^\dagger)\right) \geq \curv_{S/\m S} (\ell)
\]
and 
\[
\injcx_A((F^{e \cdot \varphi})_* N) = \cx_A\left((F^{e \cdot \varphi})_* (N^\dagger)\right) \geq \cx_{S/\m S} (\ell)
\]
Moreover, as in $\ref{relfrobcurvcx}$, we have equality when $N^\dagger$ has finite flat dimension over $R$. This occurs if and only if $N$ has finite injective dimension over $R$.
\end{proof}

\bibliographystyle{amsplain}
\enlargethispage{1\baselineskip}
\bibliography{LocalBib}

@Preamble{
 "\def \noopsort #1{}"
}

@article {AHIY12,
    AUTHOR = {Avramov, Luchezar L. and Hochster, Melvin and Iyengar,
              Srikanth B. and Yao, Yongwei},
     TITLE = {Homological invariants of modules over contracting
              endomorphisms},
   JOURNAL = {Math. Ann.},
  FJOURNAL = {Mathematische Annalen},
    VOLUME = {353},
      YEAR = {2012},
    NUMBER = {2},
     PAGES = {275--291},
      ISSN = {0025-5831,1432-1807},
   MRCLASS = {13D05 (13A35 13D02 13D07 13D09)},
  MRNUMBER = {2915536},
MRREVIEWER = {Irena\ Swanson},
       DOI = {10.1007/s00208-011-0682-z},
       URL = {https://doi.org/10.1007/s00208-011-0682-z},
}

@article {Avr96,
    AUTHOR = {Avramov, Luchezar L.},
     TITLE = {Modules with extremal resolutions},
   JOURNAL = {Math. Res. Lett.},
  FJOURNAL = {Mathematical Research Letters},
    VOLUME = {3},
      YEAR = {1996},
    NUMBER = {3},
     PAGES = {319--328},
      ISSN = {1073-2780},
   MRCLASS = {13D03 (13D05)},
  MRNUMBER = {1397681},
MRREVIEWER = {Alex\ Martsinkovsky},
       DOI = {10.4310/MRL.1996.v3.n3.a3},
       URL = {https://doi.org/10.4310/MRL.1996.v3.n3.a3},
}

@article {Kun69,
    AUTHOR = {Kunz, Ernst},
     TITLE = {Characterizations of regular local rings of characteristic
              {$p$}},
   JOURNAL = {Amer. J. Math.},
  FJOURNAL = {American Journal of Mathematics},
    VOLUME = {91},
      YEAR = {1969},
     PAGES = {772--784},
      ISSN = {0002-9327,1080-6377},
   MRCLASS = {13.95},
  MRNUMBER = {252389},
MRREVIEWER = {M.\ Nagata},
       DOI = {10.2307/2373351},
       URL = {https://doi.org/10.2307/2373351},
}

@article {AIM06,
    AUTHOR = {Avramov, Luchezar L. and Iyengar, Srikanth and Miller,
              Claudia},
     TITLE = {Homology over local homomorphisms},
   JOURNAL = {Amer. J. Math.},
  FJOURNAL = {American Journal of Mathematics},
    VOLUME = {128},
      YEAR = {2006},
    NUMBER = {1},
     PAGES = {23--90},
      ISSN = {0002-9327,1080-6377},
   MRCLASS = {13D07},
  MRNUMBER = {2197067},
MRREVIEWER = {Lars\ Winther\ Christensen},
       URL = {http://muse.jhu.edu/journals/american_journal_of_mathematics/v128/128.1avramov.pdf},
}

@book {CFH24,
    AUTHOR = {Christensen, Lars Winther and Foxby, Hans-Bj{\o}rn and Holm,
              Henrik},
     TITLE = {Derived category methods in commutative algebra},
    SERIES = {Springer Monographs in Mathematics},
 PUBLISHER = {Springer, Cham},
      YEAR = {[2024] \copyright 2024},
     PAGES = {xxiii+1119},
      ISBN = {978-3-031-77452-2; 978-3-031-77453-9},
   MRCLASS = {13D09},
  MRNUMBER = {4890472},
       DOI = {10.1007/978-3-031-77453-9},
       URL = {https://doi.org/10.1007/978-3-031-77453-9},
}

@book {BH98,
    AUTHOR = {Bruns, Winfried and Herzog, J{\"u}rgen},
     TITLE = {Cohen-{M}acaulay rings},
    SERIES = {Cambridge Studies in Advanced Mathematics},
    VOLUME = {39},
 PUBLISHER = {Cambridge University Press, Cambridge},
      YEAR = {1998},
      edition={Revised},
     PAGES = {xii+403},
      ISBN = {0-521-41068-1},
   MRCLASS = {13H10 (13-02)},
  MRNUMBER = {1251956},
MRREVIEWER = {Matthew\ Miller},
}

@article {Rad92,
AUTHOR = {Radu, Nicolae},
TITLE = {Une classe d'anneaux noeth\'eriens},
JOURNAL = {Rev. Roumaine Math. Pures Appl.},
FJOURNAL = {Revue Roumaine de Math\'ematiques Pures et Appliqu\'ees.
Romanian Journal of Pure and Applied Mathematics},
VOLUME = {37},
YEAR = {1992},
NUMBER = {1},
PAGES = {79--82},
ISSN = {0035-3965},
MRCLASS = {13E05 (13F40 13H05)},
MRNUMBER = {1172271},
MRREVIEWER = {Takashi\ Harase},
}

@article {And93,
AUTHOR = {Andr\'e, Michel},
TITLE = {Homomorphismes r\'eguliers en caract\'eristique {$p$}},
JOURNAL = {C. R. Acad. Sci. Paris S\'er. I Math.},
FJOURNAL = {Comptes Rendus de l'Acad\'emie des Sciences. S\'erie I. Math\'ematique},
VOLUME = {316},
YEAR = {1993},
NUMBER = {7},
PAGES = {643--646},
ISSN = {0764-4442},
MRCLASS = {13D03 (13A35)},
MRNUMBER = {1214408},
MRREVIEWER = {Marco\ Fontana},
}

@Inproceedings{Avr98, 
author ={Avramov, Luchezar L.}, 
title ={{Infinite free resolutions}}, 
booktitle={Six lectures on commutative algebra}, 
editor ={J. Elias and J. M. Giral and R. M. Mir\'o-Roig and S. Zarzuela}, series =PIM, 
volume =166, 
publisher=Birk, 
year =1998 }

@article {McD25,
    AUTHOR = {McDonald, Peter M.},
     TITLE = {Homological properties of the relative {F}robenius morphism},
   JOURNAL = {Proc. Amer. Math. Soc.},
  FJOURNAL = {Proceedings of the American Mathematical Society},
    VOLUME = {153},
      YEAR = {2025},
    NUMBER = {12},
     PAGES = {5013--5026},
      ISSN = {0002-9939,1088-6826},
   MRCLASS = {13A35 (13B10 13D05 13H10)},
  MRNUMBER = {4989621},
       DOI = {10.1090/proc/17182},
       URL = {https://doi.org/10.1090/proc/17182},
}

@article {Avr99,
    AUTHOR = {Avramov, Luchezar L.},
     TITLE = {Locally complete intersection homomorphisms and a conjecture
              of {Q}uillen on the vanishing of cotangent homology},
   JOURNAL = {Ann. of Math. (2)},
  FJOURNAL = {Annals of Mathematics. Second Series},
    VOLUME = {150},
      YEAR = {1999},
    NUMBER = {2},
     PAGES = {455--487},
      ISSN = {0003-486X,1939-8980},
   MRCLASS = {13D03 (13H10 14M10)},
  MRNUMBER = {1726700},
MRREVIEWER = {Paul\ Roberts},
       DOI = {10.2307/121087},
       URL = {https://doi.org/10.2307/121087},
}

@article {Hir64,
    AUTHOR = {Hironaka, Heisuke},
     TITLE = {Resolution of singularities of an algebraic variety over a
              field of characteristic zero. {I}, {II}},
   JOURNAL = {Ann. of Math. (2)},
  FJOURNAL = {Annals of Mathematics. Second Series},
    VOLUME = {79},
      YEAR = {1964},
     PAGES = {109--203; {\ 79 (1964), 205--326}},
      ISSN = {0003-486X},
   MRCLASS = {14.18},
  MRNUMBER = {199184},
MRREVIEWER = {Jean\ Giraud},
       DOI = {10.2307/1970547},
       URL = {https://doi.org/10.2307/1970547},
}

@article {Dum96,
    AUTHOR = {Dumitrescu, Tiberiu},
     TITLE = {Regularity and finite flat dimension in characteristic
              {$p>0$}},
   JOURNAL = {Comm. Algebra},
  FJOURNAL = {Communications in Algebra},
    VOLUME = {24},
      YEAR = {1996},
    NUMBER = {10},
     PAGES = {3387--3401},
      ISSN = {0092-7872,1532-4125},
   MRCLASS = {13A35},
  MRNUMBER = {1402567},
MRREVIEWER = {Ian\ M.\ Aberbach},
       DOI = {10.1080/00927879608825755},
       URL = {https://doi.org/10.1080/00927879608825755},
}

\end{document}